\documentclass[10pt,twoside, a4paper, english, reqno]{amsart}
\usepackage[dvips]{epsfig}
\usepackage{amscd}
\usepackage{amssymb}
\usepackage{amsthm}
\usepackage{amsmath}
\usepackage{latexsym}

\usepackage{upref}
\usepackage{hyperref}
\usepackage{color}
\theoremstyle{plain}
\newtheorem{thm}{Theorem}[section]
\theoremstyle{plain}
\newtheorem{lem}[thm]{Lemma}
\newtheorem{prop}[thm]{Proposition}

\theoremstyle{definition}

\newtheorem*{rem}{Remark}

\DeclareMathOperator{\sech}{sech}
\newcommand{\De} {\Delta}
\newcommand{\la} {\lambda}
\newcommand{\bn}{\mathbb{B}^{N}}
\newcommand{\rn}{\mathbb{R}^{N}}

\numberwithin{equation}{section} \allowdisplaybreaks

	\title[]{Nondegeneracy of Positive solutions of a Semilinear elliptic problem in the Hyperbolic space}

\date{}

\author[K Sandeep]{K Sandeep}
\address[K Sandeep]{\newline
  Centre for Applicable Mathematics,
  Tata Instiute of Fundamental Research,
   P.O.\ Box 6503, GKVK Post Office,
   Bangalore 560065, India}
 \email[]{sandeep@math.tifrbng.res.in}

 \author[Debdip Ganguly]{Debdip Ganguly}
 \address[Debdip Ganguly]{\newline
  Centre for Applicable Mathematics,
  Tata Instiute of Fundamental Research,
   P.O.\ Box 6503, GKVK Post Office,
   Bangalore 560065, India}
 \email[]{debdip@math.tifrbng.res.in}

\begin{document}

\begin{abstract}
In this article we will study the nondegeneracy properties of positive finite energy solutions of 
the equation $-\De_{\bn} u - \la u = |u|^{p-1}u $ in the Hyperbolic space. We will show that the degeneracy occurs only in an 
$N$ dimensional subspace. We will prove that the positive solutions are nondegenerate in the case of geodesic balls.
\end{abstract}

\maketitle

\section{Introduction}
 
 In this article we will study the degeneracy properties of positive solutions of the problem
\begin{equation} \label{E:1.1}
 -\De_{\bn} u - \la u = |u|^{p-1}u ,   u \in H^{1}(\bn)
\end{equation}
where $\la < (\frac{N-1}{2})^{2}$ , $1 <p \leq \frac{N+2}{N-2}$ if $N\geq 3$ , $1<p < \infty $ if $N = 2 $ and ${H^{1}}(\bn)$ 
denotes the Sobolev
space on the disc model of the Hyperbolic space $\bn$
 and  $\De_{\bn}$  denotes the Laplace Beltrami operator on $\bn$. Apart from its own mathematical interest, this equation arises in the 
study of Grushin operators(\cite{B}), Hardy-Sobolev-Mazy'a equations(\cite{HS},\cite{MS}) and in some fluid dynamics models (\cite{MT}). 
This equation is also the hyperbolic version
 of the well known Brezis Nirenberg problem (\cite{BN}).\\\\
Existence and uniqueness of positive solutions of \eqref {E:1.1} has been studied in \cite{MS}. In the subcritical case 
(i.e., $p> 1$ if $N =2$ and $1<p < 2^{*} -1 $ if $N \geq 3$) the equation \eqref {E:1.1} has a positive solution iff
 $\la <\left(\frac{N-1}{2}\right)^2$. In the critical case  $N\ge 4, p = 2^{*} -1 $, \eqref {E:1.1} has a positive solution iff $\frac{N(N-2)}{4}<\la <\left(\frac{N-1}{2}\right)^2$ while for $N=3$, positive solutions does not exist for any $\la$. These positive solutions are also shown to be unique up to Hyperbolic isometries except $N=2$ and $\la > \frac{2(p+1)}{(p+3)^{2}}$.
Exitence of infinitely many sign changing radial solutions in higher dimensions has been obtained in \cite{PS} and \cite{DS}.\\\\
Our main aim in this paper is to study the degeneracy properties of positive solutions obtained in \cite{MS}.
It is well known that the degeneracy properties are crucial in studying the stability of the solution and also to obtain various existence 
results, bifurcation results, etc using tools like finite dimensional reduction. In the case of bounded domains in Euclidean space
degeneracy of solution is used to study the uniqueness and other qualitative properties of solutions (see \cite{P},\cite{DGP},\cite{S}).\\\\
For some special values of $\la$ and $p$, the degeneracy of the solution has been studied in \cite{HS}. In this special case the solution 
of \eqref{E:1.1} is explicitly known and this plays a major role in the proof as one can use informations on the solution of a hyper geometric
 ODE arising in the analysis. However in general we do not have explicit expression for the solution of \eqref{E:1.1} and the proof presented
 in the special case in \cite{HS} does not carry over.\\\\
In view of the existence and non-existence results obtained in \cite{MS} we assume: 
\begin{equation}\label{conditions}
\left.
\begin{array}{lll}
\la <(\frac{N-1}{2})^{2} & {\rm when} &1 <p < \frac{N+2}{N-2},N\ge 3 \\
\frac{N(N-2)}{4}<\la <(\frac{N-1}{2})^{2}  & {\rm when} &p = \frac{N+2}{N-2}, N\ge 4\\ 
\la < \frac{2(p+1)}{(p+3)^{2}} & {\rm when} & N = 2 .
\end{array} \right\}
\end{equation}
 Let $U$ be the unique positive radial solution of \eqref{E:1.1} under the assumption \eqref{conditions}.
Consider the linearized operator $L$ given by
\begin{equation}\label{E:2.4}
   L (\Phi) = - \De_{\bn} \Phi - \la \Phi - p U^{p-1} \Phi ,\;\; \Phi \in {H^{1}}(\bn)
\end{equation}
We are interested in knowing whether $0$ is in the spectrum of $L$ and more generally to find the kernel of $L$.\\
Degeneracy properties of similar equations in the Euclidean space $\rn$ are well known. In fact if $u$ is a solution then 
a differentiation of the equation will tell us that $\frac{\partial u}{\partial x_i}$ is a solution of the linearized equation.
We show that in the hyperbolic case a different vector field will do the job $\frac{\partial }{\partial x_i}$ does in $\rn$.  
Observe that the equation is invariant under hyperbolic isometries, i.e, if $u$ solves \eqref{E:1.1} and $\tau$ is an isometry of $\bn$ 
then $u\circ \tau$ also solves 
\eqref{E:1.1}. This fact indicates that the solutions has to be degenerate. However we see that the degeneracy happens only along an $N$ 
dimensional subspace. 
\begin{thm}\label{T:3.1} Let $V_i,\;i=1,...,N$ be the vector fields in $\bn$ given by
\begin{equation}
V_i(x) = (1+|x|^2)\frac{\partial}{\partial{x_i}} - 2x_i\sum\limits_{j=1}^N x_j\frac{\partial}{\partial{x_j}},\;\;i=1,...,N
\end{equation}
and $\Phi_{i}(x) = V_i\left( U \right)$, then $L (\Phi_i)=0$ for all $i=1,...,N.$ Moreover 
$\{\Phi_{i}\;:\;i=1,...,N\}$ is a basis for the kernel of $L.$
\end{thm}

Now consider the problem in geodesic balls. Let $B$ be a geodesic ball in $\bn$, consider the problem
\begin{equation}\label{eqn-ball}
\left.
\begin{array}{rlll}
-\De_{\bn} u - \la u &= &|u|^{p-1}u &{\rm in} \; B\\
u &>& 0 &{\rm in} \; B\\
u &=& 0 &{\rm on} \; \partial B
\end{array} \right\}
\end{equation}
We know from \cite{ST} that this problem has a solution under the assumptions \eqref{conditions}.
Using moving plane method (see \cite{ADG},\cite{MS} and the references therein) we know that any solution of \eqref{eqn-ball} is radial and the 
uniqueness of radial solution has been established in \cite{MS}. Let us denote by $u$ the unique positive radial solution of 
\eqref{eqn-ball} and $\tilde L$ the linearized operator
 \begin{equation}\label{lin-ball}
  \tilde L (\Phi) = - \De_{\bn} \Phi - \la \Phi - p u^{p-1} \Phi ,\;\; \Phi \in H^1_0(B)
\end{equation} 
In the case of Euclidean ball, the nondegeneracy of the unique positive solution of semilinear elliptic equations has been studied
 by various authors (see \cite{DGP},\cite{P} and the references therein ). In the hyperbolic case we prove :\\ 
\begin{thm}\label{nondegeneracy-b} Under the assumptions \eqref{conditions}, the unique solution of \eqref{eqn-ball} is 
nondegenerate.
\end{thm}
\begin{rem} Both Theorem \ref{T:3.1} and Theorem \ref{nondegeneracy-b} extends to the case $\la = \left(\frac{N-1}{2}\right)^2$, but in the first case
we have to work with the energy space  introduced in \cite{MS} instead of $H^1(\bn)$.\\\\
\end{rem}
\section{Notations and Preliminaries}
In this section we will introduce some of the notations and definitions used in this paper and also recall some of the embeddings
 related to the Sobolev space in the hyperbolic space.\\\\
We will denote by $\bn$ the disc model of the hyperbolic space, i.e., the unit disc equipped with 
the Riemannian metric $g = \sum\limits_{i=1}^N(\frac{2}{1-|x|^2})^2dx_i^2$.\\
The corresponding volume element is given by $dv_{\bn} = (\frac{2}{1-|x|^2})^Ndx $.  
The hyperbolic gradient $\nabla_{\bn}$ and the hyperbolic Laplacian $\De_{\bn}$ are given by
$$
 \nabla_{\bn}=(\frac{1-|x|^2}{2})^2\nabla,\ \ \  \De_{\bn}=(\frac{1-|x|^2}{2})^2\De+(N-2)\frac{1-|x|^2}{2}<x,\nabla>
$$\\
{\bf Sobolev Space.} We will denote by ${H^{1}}(\bn)$ the Sobolev space on the disc model of the Hyperbolic space $\bn$ and 
 $H^{1}_r(\bn)$ denotes the subspace of $H^{1}(\bn)$ consisting of radial functions. For a geodesic ball $B$ in $\bn$ we will
 denote by $H^1_0(B)$ the completion of $C_c^\infty(B)$ with the norm $||u||^2 = \int\limits_{B} |\nabla_{\bn}  u|^2$ and 
$H^1_{0,r}(B)$ its subspace consisting of radial functions.\\\\
Let us recall the following Sobolev embedding theorem (see \cite{MS} for a proof):\\\\
{\bf Poincar\'e-Sobolev inequality:}
For every $N \geq 3$ and every $p\in (1, \frac{N+2}{N-2}] $ there is a constant $S_{N,p}>0$ such that
\begin{equation}\label{PS2}
 S_{N,p} \left(\int\limits_{\bn}  |u|^{p+1} dv_{\bn} \right)^{\frac{2}{p+1}} \leq 
\int\limits_{\bn} \left[ |\nabla_{\bn}  u|^2  - {\frac{(N-1)^2 }{4}} u^2 \right]  dv_{\bn} \quad 
\end{equation}
for every $ u\in H^{1}(\bn). \: \: $ If $N=2$ any $p> 1$ is allowed.\\\\
In general the embedding $H^{1}(\bn) \hookrightarrow L^{p+1}(\bn)$ is not compact, however for radial functions, using the point 
wise estimate we have the following compactness lemma (see \cite{PS} for a proof):\\
\begin{lem}\label{cpt-embedding}The embedding
$H^{1}_r(\bn) \hookrightarrow L^{p+1}(\bn)$ is compact if $1<p< \frac{N+2}{N-2}$ and $N\ge 3$ or $1<p<\infty$ and $N=2$.
\end{lem}
For a positive function $V$ on $\bn$ define: $$L^{2}(\bn , V dv_{\bn})= \left\{ u : \int_{\bn}u^2\;Vdv_{\bn}<\infty \right\}$$ 
then we have the following compact embedding in the general case:\\
\begin{prop}\label{P:3.2}
 The Embedding $H^{1}(\bn) \hookrightarrow L^{2}(\bn , V dv_{\bn})$ is compact if $V(x)\rightarrow 0$ as $x\rightarrow \infty$. 
\end{prop}

\begin{proof}
 
Let $\{u_{n}\}_{n} \in H^{1}(\bn)$ be such that $||u_{n}||_{H^{1}(\bn)} \leq C$ for some $C>0$. Then upto a subsequence
$u_{n}\rightharpoonup u$ in $H^{1}(\bn)$  and pointwise. To complete the proof we need to show $u_{n} \longrightarrow u$
 in $L^{2}(\bn, V dv_{\bn})$. Let $\epsilon >0$, since $V(x) \rightarrow 0$ as $|x| \rightarrow \infty$ 
 there exists $R_{\epsilon} < 1$ such that $V(x) < \epsilon$ for all $|x| \geq R_{\epsilon}$ . Thus
\begin{eqnarray*}
 \left| \int_{\bn} u_{n}^{2} V dv_{\bn} - \int_{\bn} u^{2} V dv_{\bn}\right| &&\leq \int_{\bn} |u_{n}^{2} - u^{2}| V dv_{\bn} \\
\end{eqnarray*}
\begin{eqnarray*}
       && =  \int_{|x|\leq R_{\epsilon}} |u_{n}^{2} - u^{2}| V dv_{\bn} + \int_{|x|\geq R_{\epsilon}} |u_{n}^{2} - u^{2}| V dv_{\bn}
\end{eqnarray*}
 The first integral converges to zero as $n\rightarrow \infty$ thanks to Rellich's compactness theorem. Using Poincar\'e-Sobolev 
inequality the second integral can be estimated  as follows:
\begin{eqnarray*}
 \int_{|x| \geq R_{\epsilon}} |u_{n}^{2} - u^{2}| V dv_{\bn} && < \epsilon \int_{|x|\geq R_{\epsilon}} |u_{n}^{2} - u^{2}| dv_{\bn}\\
&& \leq \epsilon\left[\int_{|x|\geq R_{\epsilon}} u_{n}^{2} dv_{\bn} + \int_{|x|\geq R_{\epsilon}} u^{2} dv_{\bn}\right]
  < C \epsilon 
\end{eqnarray*}
Since $\epsilon$ is arbitrary, this implies $u_{n} \rightarrow u$  strongly in $L^{2}({\bn, V dv_{\bn}}).$
\end{proof}
We also need some information on the isometries of $\bn$. Below we recall the definition of a special type of isometries namely the 
hyperbolic translations. For more details on the isometry group of $\bn$ we refer to \cite{JR}.\\
{\bf Hyperbolic translations} For $b\in \bn$ define
\begin{equation}
 \tau_{b}(x) = \frac{(1 - |b|^{2})x + (|x|^{2} + 2 x.b + 1)b}{|b|^{2}|x|^{2} + 2x.b +1}
\end{equation}
then $\tau_{b}$ is an isometry of $\bn$ with $\tau_{b}(0) = b$. The map $\tau_{b}$ is called the hyperbolic translation of $\bn$ by $b$ .  
\section{Radial Nondegeneracy} In this section we prove the radial nondegeneracy of the linearized operators. First consider the case of
 \eqref{E:1.1}:
\begin{thm}\label{nondegeneracy-r} Let $\Phi \in H^{1}_r(\bn)$ satisfies $L(\Phi)=0$ then $\Phi =0.$
\end{thm}
Similarly for \eqref{eqn-ball} we have:
\begin{thm}\label{nondegeneracy-r-b} Let $\Phi \in H^1_{0,r}(B)$ satisfies $\tilde L(\Phi)=0$ then $\Phi =0.$
\end{thm}
We will prove Theorem \ref{nondegeneracy-r}, the proof of Theorem \ref{nondegeneracy-r-b} follows similarly with obvious modifications. 
As in \cite{US} we use the information on the Morse index of a perturbed problem to establish the result.\\
\begin{thm}\label{morseindex} Let $U$ be the unique positive radial solution of \eqref{E:1.1}, and consider for 
$\epsilon >0, I_{\epsilon}$ defined on $H^{1}_r(\bn)$ by 
\begin{equation}
I_{\epsilon}(u) = \frac{1}{2} \int_{\bn} \left[|\nabla_{\bn} u|^{2} - (\la-\epsilon U^{p-1}) u^2\right] - 
\frac{1+\epsilon}{p+1} \int_{\bn} |u^+|^{p+1}
\end{equation}
then there exists an $\epsilon >0$ such that $u=U$ is the unique critical point of $I_{\epsilon}, U$ is a mountain pass critical point
and hence of radial Morse index one.
\end{thm}
We postponed the proof of this theorem to Section 5.\\
{\bf Proof of Theorem \ref{nondegeneracy-r}}. Assume that there exists a $\Phi \not= 0$ such that $L(\Phi) =0.$ We will show that this will
contradict the fact that $U$ is a  critical point of $I_{\epsilon}$ with radial Morse index of one.\\
Recall the Morse Index of $U$ denoted $i (I_{\epsilon}, U)$ is defined as
\[
 i (I_{\epsilon}, U) = \mbox{max} \{\mbox{dim} H : H \subset H^{1}(\bn) \mbox{is a subspace such that}  I_{\epsilon}^{\prime\prime}[U](h, h) < 0\ \forall h \in H \backslash \{0\}\}
\]
We have 
$$I_{\epsilon}^{\prime\prime}[U](v,v) = \int_{\bn} [|\nabla_{\bn} v|^{2} - (\la-\epsilon U^{p-1}) v^2]- (1+\epsilon)p \int_{\bn}U^{p-1} v^2$$
Then $$I_{\epsilon}^{\prime\prime}[U](U,U) = (1+\epsilon)(1-p)\int_{\bn}U^{p+1} <0 \;{\rm and}\; 
I_{\epsilon}^{\prime\prime}[U](\Phi,\Phi) =\epsilon(1-p)\int_{\bn}U^{p-1}\Phi^2 <0.$$\\ 
Since $$\int_{\bn}[\nabla_{\bn}U \cdot \nabla_{\bn}\Phi- \la U \Phi] = \int_{\bn}U^{p}\Phi =0 $$ 
we get for real numbers $ \alpha,\beta$
$$I_{\epsilon}^{\prime\prime}[U](\alpha U+\beta \Phi,\alpha U+\beta \Phi) = \alpha^2I_{\epsilon}^{\prime\prime}[U](U,U) +
\beta^2I_{\epsilon}^{\prime\prime}[U](\Phi,\Phi) <0$$
Moreover $U$ and $\Phi$ are linearly independent and hence the radial Morse index of $U$ is at least $2$, which is a contradiction.

\section{Proof of Theorems} 
In this section we will prove our main theorems.\\
{\bf Proof of Theorem \ref{T:3.1}} We divide the proof in to several steps.\\\\
{\bf Step 1} Let $\Phi_{i}$ be as define in Theorem \ref{T:3.1}, then $L(\Phi_{i})=0$ for all $i.$\\\\
{\it Proof of Step 1.} Let $x \in \bn$ then $V_i(x) =\gamma^\prime(0)$
where $\gamma(t) = \tau_{t e_{i}}(x),\;|t|<1$, where $\{e_{i}\}$ is the standard basis of $\rn$ and $\tau_{t e_{i}}$ is the hyperbolic translation.
Thus $\Phi_{i}(x) = \left[ \frac{d}{dt}|_{t=0} U\circ \tau_{t e_{i}}(x)\right]$. Recall that if $U$ is a solution of \ref{E:1.1},
then so is $V = U\circ \tau$ for any isometry $\tau$ of $\bn$. In particular
\begin{equation}\label{E:2.6}
 -\De_{\bn} (U\circ \tau_{te_i}) - \la U\circ\tau_{te_i} = (U\circ\tau_{te_i})^{p}
\end{equation}
Differentiating with respect to $t$ and evaluating at $t=0$ we see that $L(\Phi_{i})=0$ where
\begin{eqnarray}\label{E:2.7}
 \Phi_{i} &&= \left[ \frac{d}{dt}|_{t=0} U(\tau_{t e_{i}})\right]=<\nabla_{\bn} U(x) , e_{i}>_{\bn} (|x|^{2} + 1) 
- 2 <\nabla_{\bn} U(x) , x>_{\bn} x_{i} \nonumber \\
&& = \frac{\partial U}{\partial x_{i}} (|x|^{2} + 1) - 2 <\nabla_{\bn} U(x) , x>_{\bn} x_{i} \nonumber\\
&& = \frac{x_{i}}{|x|} [1 - |x|^{2}] U^{\prime}(|x|)
\end{eqnarray}
From the estimates on $U$ (established in \cite{MS}) we get $\Phi_{i} \in H^{1}(\bn),\;i = 1,\ldots, N$ and are  linearly independent.\\\\ 
{\bf Step 2.} Let $\bn_{+} = \{ x\in \bn : x_1>0\}$ and $\psi \in H^1_0(\bn_{+})$ satisfies 
\begin{equation}\label{halfball}
-\De_{\bn}\psi - \la \psi = p U^{p-1} \psi \;\mbox{in} \; \bn_{+} 
\end{equation}
then there exists a constant $c\in\mathbb{R}$ such that $\psi = c\Phi_1.$\\
{\it Proof of Step 2.}We know from Lemma \ref{P:3.2} that the embedding $H^{1}_0(\bn_{+}) \hookrightarrow L^{2}(\bn_{+}, U^{p-1} dv_{\bn})$
 is compact.Thus by standard compact operator theory and maximum principle the weighted eigenvalue problem 
$$-\De_{\bn}\psi - \la \psi = \mu U^{p-1} \psi \;\mbox{in} \; \bn_{+},\;\; \psi \in H^1_0(\bn_{+})$$
has a discrete set of eigenvalues $\mu_1<\mu_2<... \rightarrow \infty$, with $\mu_1$ is simple. Moreover any eigenfunction corresponding 
to $\mu_i,\;i>1$ change sign. Since $ \Phi_{1}|_{\bn_{+}}$ is an eigenfunction with $\mu=p$ and it has a constant sign (since $U^{\prime}<0$)
 we conclude that
$\mu_1=p$ and hence Step 2 follows from the simplicity of $\mu_1=p.$\\\\
{\bf Step 3.} Let $\Phi\in H^1(\bn)$ and satisfies $L(\Phi)=0$, then for any unit vector $\nu \in \rn$ there exists a constant 
$c \in \mathbb{R}$ such that
\[
 \Phi (x) - \Phi(R_{\nu}(x)) = c \Phi_{1} (O_{\nu}(x)) \hspace{.5 cm} \forall \hspace{.25 cm} x \in \bn
\]
where $R_{\nu}$ and $O_{\nu}$ denote the reflection with respect to the hyperplane $\{(\xi,\nu) = 0\}$  
and an orthogonal transformation satisfying $O_{\nu} (\nu) = e_{1}$ respectively.\\
{\it Proof of Step 3.} Let us first consider the case when $\nu = e_{1}= (1, \ldots, 0) \in \rn$ then 
$\psi(x) = \Phi(x) - \Phi(R_{e_{1}}x)$ satisfies \eqref{halfball} and hence from Step 2 we get the existence of
 $c\in \mathbb{R}$ such that 
\[
 \psi(x) = c \Phi_{1}(x)
\]
Also note that $\Phi_{1}(x) = \Phi_{1}(O_{e_1}x)$, hence we have 
\[
 \Phi(x)  - \Phi(R_{e_1} x) = c \Phi_{1} (O_{e_1}x)
\]
Similar argument will give for any $\nu$ unit vector $\in \rn$ and $O_{\nu}$ denote the orthogonal transformation such that $O_{\nu}\nu =e_{1}$ 
\[
 \Phi(x) - \Phi(R_{\nu} x) = c \Phi_{1} (O_{\nu} x)
\]
This completes the proof of Step 3.\\
{\bf Step 4.} Let $\Phi$ be in the kernel of $L$, then there exist $c_{1}, \ldots , c_{N}\in\mathbb{R}$ such that the function $\Phi^{r}$
 defined by
\[
 \Phi^{r} := \Phi - \sum_{i=1}^{N} c_{i} \Phi_{i}
\]
is radial.\\
{\it Proof of Step 4}
 We have from Step 3 , given any unit vector $\nu\in \rn$ , there exist $c \in \mathbb{R}$ such that 
\begin{equation}\label{E:3.2}
 \Phi(x) - \Phi(R_{\nu}x) = c \Phi_1(O_{\nu} x)
\end{equation}
We earlier remark that for the orthogonal transformation $O_{e_{1}}$ we have $\Phi_{1}(x) = \Phi_{1}(O_{1}x)$.

Let $R_{i} := R_{e_{i}}$ , $O_{i} = O_{e_{i}}$ then have from \ref{E:3.2}
\[
 \Phi(x) - \frac{1}{2} [\Phi(x) + \Phi(R_{1}x)] = c_{1} \Phi_{1} (O_{1}x) = c_{1} \Phi_{1}(x)
\]
This imply 

\[ 
 \Phi(x) - c_{1}\Phi_{1}(x) = \frac{1}{2} [\Phi(x) + \Phi (R_{1} x)]
\]
Hence $\Phi(x) - c_{1}\Phi_{1}(x)$ is even in $x_{1}$.

Let $\Phi^{1} = \Phi - c_{1}\Phi_{1}$ and apply \ref{E:3.2} we have 
\[
 \Phi^{1}(x) - \frac{1}{2}[\Phi^{1}(x) + \Phi^{1}(R_{2} x)]  = c_{2} \Phi_{1} (O_{2}x)
\]
From \eqref{E:2.7} we get  $\Phi_{1}(O_{2}x)= \Phi_{2} (x)$ :

This imply
\[
 \Phi^{1} - c_{2} \Phi_{2}(x) = \frac{1}{2}[\Phi^{1}(x) + \Phi^{1}(R_{2}x)]
\]
Hence $\Phi - c_{1}\Phi_{1}(x) - c_{2}\Phi_{2}(x)$ is even in $x_{2}$ and of course, in $x_{1}$ as well. 

By iterating, we conclude that, for some $c_{1}, \ldots , c_{N}$ , $\Phi^{r} : = \Phi - \sum_{i=1}^{N} c_{i}\Phi_{i}$ is even all the 
$x_{i}$ and hence $\nabla_{\bn} \Phi^{r}(0) = 0$. 

By \ref{E:3.2}, $ \Phi^{r}(x) - \Phi^{r} (R_{\nu}x) = c\Phi_{1}(O_{\nu}x)$   for some $c$. Since by Hopf lemma 
\[
 \nabla_{\bn}\Phi_{1}(O_{\nu}x)|_{x= 0} \neq 0
\]
then we must have $c = 0$, i.e. $\Phi^{r}(x) = \Phi^{r}(R_{\nu}x)$ for all $x$ and for all $\nu$. Hence $\Phi^{r}$ is radial.\\
{\bf Step 5.} It follows from Theorem \ref{nondegeneracy-r} that $\Phi^{r}=0$ and hence $\Phi = \sum_{i=1}^{N} c_{i} \Phi_{i}.$ This 
completes the proof of Theorem \ref{T:3.1}\\\\
{\bf Proof of Theorem \ref{nondegeneracy-b}.} Thanks to Theorem \ref{nondegeneracy-r-b}, we only have to show that if $\tilde L(\Phi)=0$
then $\Phi$ is radial.\\\\
Let $u(x) =\tilde u(|x|)$ be the unique radial solution of \eqref{eqn-ball}, define $v(x)= \frac{x_{i}}{|x|} [1-|x|^{2}]{\tilde u}^{\prime}(|x|)$, then 
we see that $\tilde L(v) = 0\;\;{\rm in}\; B^+$ where $B^+ = \{ x\in B : x_1>0 \}$. Moreover  $v<0$ in $B^+$ and $v<0$ on 
$\partial B^+ \cap \partial B$, hence the first eigenvalue of $\tilde L$ in $B^+$ is positive (see \cite{DGP}).\\\\  
Let $\Phi \in H^1_0(B)$ be such that $\tilde L(\Phi)=0$, then from standard elliptic theory we know that $\Phi$
is smooth in $\overline B.$ Let $\psi(x) = \Phi(x)-\Phi(\tilde x)$ where $\tilde x =(-x_1,x_2,...,x_n)$ then $\psi$ solves
\begin{equation}
 \tilde L (\psi) =0 \;\;{\rm in}\; B^+,\; \psi \in H^1_0(B^+)
\end{equation}
Since the first eigen value of $\tilde L$ is positive we get $\psi =0$ and hence $\phi$ is symmetric in the $x_1$ direction. 
Since the equation is invariant under rotation we get $\Phi$ is symmetric in all the directions and hence radial. This completes the proof 
of Theorem \ref{nondegeneracy-b}.

\section{Perturbed problems}
In this section we will prove Theorem \ref{morseindex}. Let $U$ be the unique solution of \eqref{E:1.1}, define for $u\in H^1_r(\bn)$
\begin{equation}\label{energy-pert}
I_{\epsilon}(u) = \frac{1}{2} \int_{\bn} \left[|\nabla_{\bn} u|^{2} - (\la-\epsilon U^{p-1}) u^2\right] - 
\frac{1+\epsilon}{p+1} \int_{\bn} |u^+|^{p+1}
\end{equation}
Then we know that $I_{\epsilon}^{\prime}(u) = 0$ iff $u$ solves the PDE
\begin{equation} \label{perturbed}
\left.
\begin{array}{rll}
 -\De_{\bn} u - (\la -\epsilon U^{p-1} )u &= &(1+\epsilon)|u|^{p-1}u\\ 
 u &> &0\\
  u &\in &H^{1}_r(\bn)
\end{array}\right\}
\end{equation}
Observe that for any $\epsilon >0$, $u=U$ itself is a solution of \eqref{perturbed}. Our aim in this section is to prove that $U$ is the only positive radial solution of this equation for small $\epsilon$ and $U$ is a mountain pass critical point.\\
\begin{thm}\label{exis-per} Let $\la$ satisfy \eqref{conditions}, then there exists an $\epsilon_0>0$ such that for all $0<\epsilon <\epsilon_0$, $I_{\epsilon}$ on $H^{1}_r(\bn)$ has a mountain pass critical point.
\end{thm}
\begin{proof} The inequality \eqref{PS2} easily establishes that $I_{\epsilon}$ has the mountain pass geometry. Also for 
$1<p<\frac{N+2}{N-2}$  $I_{\epsilon}$ 
satisfies the Palais Smale condition thanks to Lemma \eqref{cpt-embedding}. Thus the theorem follows in the subcritical case. \\
In the critical case, in general $I_{\epsilon}$ does not satisfy the P.S. condition. However it follows from Theorem 3.3 of \cite{PS} that
$I_{\epsilon}$ does not satisfy the Palais Smale condition at level $\beta$ if $0<\beta<\frac{S^{\frac{N}{2}}}{N}$ where $S$ is the best 
constant in the Euclidean Sobolev inequality given by 
$$ S= \inf\left\{\int\limits_{\rn}|\nabla \phi|^2\;:\; \phi \in C_c^\infty(\rn),\int\limits_{\rn}|\phi|^{2^\ast}=1\right\}.$$  
Thus in the critical case it remains to find a function $u_1 \in H^{1}_r(\bn) $ such that $I_{\epsilon}(u_1) \le 0$ and 
$\sup\limits_{0\le t\le 1}I_{\epsilon}(tu_1)<\frac{S^{\frac{N}{2}}}{N}.$ Let us assume that $\epsilon$ is small enough and 
$\la-\epsilon U^{p-1} \ge \tilde \la > \frac{N(N-2)}{4}.$ Then for $u\in C_c^\infty(\bn)$, after a conformal change of metric, i.e., 
putting $v(x) = \left(\frac{2}{1-|x|^2}\right)^{\frac{N-2}{2}}u(x)$ we get
$$I_{\epsilon}(u) = \frac{1}{2} \int_{\bn} \left[|\nabla v|^{2} - a(x) v^2\right]dx- 
\frac{1+\epsilon}{p+1} \int_{\bn} |v^+|^{2^{\ast}}dx :=J_{\epsilon}(v)$$
where $a(x)= \left(\la-\frac{N(N-2)}{4}-\epsilon U^{p-1}\right)\frac{4}{(1-|x|^2)^2}$. Since $ a(x)>0$ from the well studied Brezis Nirenberg problem (\cite{BN}) we know the 
existence of $v_1 \in C_c^\infty(\bn)$ such that $J_{\epsilon}(v_1) \le 0$ and 
$\sup\limits_{0\le t\le 1}J_{\epsilon}(tv_1)<\frac{S^{\frac{N}{2}}}{N}.$ Thus $u_1(x) = \left(\frac{1-|x|^2}{2}\right)^{\frac{N-2}{2}}v_1(x)$
satisfies all the required properties. Thus $I_{\epsilon}$ has a radial mountain pass critical point.\end{proof}
\begin{thm}\label{uni-per} Let $\la$ satisfy \eqref{conditions}, then there exists an $\epsilon_0>0$ such that for all $0<\epsilon <\epsilon_0$, \eqref{perturbed} has a unique positive radial solution.
\end{thm}
\begin{proof}First note that any positive radial solution of \eqref{perturbed} is radially decreasing. Denote the solution by $u(|x|).$ Normalizing the 
constant on the RHS to one and writing the equation in hyperbolic polar co-ordinates, i.e., $|x|= \tanh \frac t2$
we have to prove the uniqueness of solutions of the ODE 
\begin{equation}\label{E:3.4}
 u^{\prime\prime} + \frac{N-1}{\tanh t} u^{\prime} + \la_{\epsilon} u + u^{p} =0, \hspace{.5cm} u^{\prime}(0) = 0,
 \end{equation}
satisfying 
\begin{equation}\label{finite-energy}
\int\limits_0^\infty[u^2 +|u^{\prime}|^2](\sinh t)^{N-1}\;dt < \infty 
\end{equation}
where $\la_{\epsilon}=\la-\epsilon U^{p-1}.$ The term on the LHS of \eqref{finite-energy} is the  $H^{1}(\bn)$ norm of $u.$ 
As mentioned before, when $\epsilon=0$ uniqueness of solutions of \eqref{E:3.4}-\eqref{finite-energy} has been established in \cite{MS}.
For $\epsilon >0$ also the proof follows in the same lines as in \cite{MS} with some modifications. So we will only outline the proof, 
modifying the details which are likely to differ in the case of $\epsilon >0$,
and refer to \cite{MS} for details.\\\\
The main ingredient in the uniqueness proof is the following auxiliary energy $E_{\hat{u}}$ (see \cite{MS},\cite{UR})defined by :
\[
 E_{\hat{u}}(t) : = \frac{1}{2} (\sinh^{\beta} t)\hat{u}'^{2} + \frac{|\hat{u}|^{p+1}}{p+1} + \frac{1}{2} G_{\epsilon}\hat{u}^{2}
\]
where $\hat{u} = (\sinh^{\alpha} t) u ,\; \alpha = \frac{2(N-1)}{p+3},\; \beta:= \alpha(p -1)$ and 
$ G_{\epsilon}(t) := A_{\epsilon}(t) \sinh^{\beta} t + B \sinh^{\beta -2} t $
\[
 A_{\epsilon} : = A - \epsilon U^{p-1}, \hspace{.5 cm}  A : =\la - \frac{\alpha^{2}(p+1)}{2}, \hspace{.25 cm} 
B := \frac{\alpha}{2}[2- \alpha (p+1)].
\]
Proceeding exactly as in the proof of Lemma 3.4 of \cite{MS} we get:\\\\
{\it Let $N \geq 2$, $p> 1$, $\la < \left(\frac{N-1}{2}\right)^{2}$ and $u > 0$ be a solution of \ref{E:3.4}. Then 
\begin{equation}\label{est}
\left.
\begin{array}{c}
  \lim\limits_{t\rightarrow \infty}\frac{\log u^{2}}{t} = \lim\limits_{t\rightarrow \infty} \frac{\log {u^{\prime}}^{2}}{t} =
\lim\limits_{t\rightarrow \infty} \frac{\log [u^{2} + {u^{\prime}}^{2}]}{t} = -[N-1 + \sqrt{(N-1)^{2} - 4 \la}] \\
\lim\limits_{t\rightarrow \infty}\frac{u^\prime(t)}{u(t)} = -\frac{[N-1 + \sqrt{(N-1)^{2} - 4 \la}]}{2}
\end{array}\right\}
\end{equation}}\\
Using the estimates in \eqref{est} we have the following estimates on the auxiliary energy :\\
\begin{equation}\label{energy0}
E_{\hat{u}}(t) = \left\{
\begin{array}{lll}
  \frac{\alpha^2+B}{2}\sinh ^{\alpha(p+1)-2}t u^2(t) +o(1) &{\rm if} &N=2 \\
 o(1) &{\rm if} & N\ge 3
\end{array}\right.
\end{equation}
as $t\rightarrow 0$, and
\begin{equation}\label{energyinf}
E_{\hat{u}}(t) = o(1) \;{\rm as} \;t\rightarrow \infty \;{\rm if} \; N\ge 2
\end{equation}
Since $\hat{u}$ satisfies 
\[
 (\sinh^{\beta} t) \hat{u}^{\prime\prime} + \frac{1}{2} [\sinh^{\beta} t]^{\prime} \hat{u}^{\prime} + G_{\epsilon}(t) \hat{u}^{p} =0
\]
we get $$\frac{d}{dt} E_{\hat{u}} (t) = \frac{1}{2} G_{\epsilon}^{\prime} \hat{u}^{2}$$
Next we claim:\\
{\it Claim.} There exists $\epsilon_0>0$ and $t_{1}>0$ such that for all $0\le\epsilon < \epsilon_0$
\begin{equation}\label{mono}
\left.
\begin{array}{lll}
 G^{\prime}_{\epsilon}(t) < 0  &\forall t > 0 &{\rm if}   N=2 \;{\rm and}\;\la < \frac{2(p+1)}{(p+3)^{2}}\\
 G^{'}_{\epsilon}(t) (t_{1} - t) > 0 &\forall  t \neq t_{1} &{\rm if}\; N \geq 3,p < 2^* -1, \la < \frac{2(N-1)^{2}(p+1)}{(p+3)^{2}}\\
 G^{'}_{\epsilon}(t) > 0  &\forall t > 0 &{\rm if} N \geq 3,p < 2^* -1,\la \geq \frac{2(N-1)^{2}(p+1)}{(p+3)^{2}} \\
 G^{'}_{\epsilon}(t) > 0 &\forall t > 0 &{\rm if} N \geq 3,p = 2^* -1,\la > \frac{N(N-2)}{4}

\end{array}\right\}
\end{equation}
{\it Proof of Claim :} By direct calculation we have
\begin{eqnarray*}\label{E:3.9}
 G_{\epsilon}^{\prime}(t) &&= A \beta (\sinh t)^{\beta - 1} \cosh t +  B (\beta -2) (\sinh t)^{\beta -3}
  \cosh t 
- \epsilon (p-1) U^{p-2} U^{\prime}(t) (\sinh t)^{\beta} \\ 
   &&- \epsilon \beta  U^{p-1} (\sinh t)^{\beta -1} \cosh t \\ 
&&= \beta (\sinh t)^{\beta -3} \cosh t \left\{ \left(A - \epsilon U^{p-1} - \frac{\epsilon(p-1)}{\beta} U^{p-2}U^{\prime}(t)\tanh t\right)
 (\sinh t)^{2}
 + \frac{B(\beta -2)}{\beta} \right\}
\end{eqnarray*}
To analyze the sign of $G_{\epsilon}^{\prime}$, first recall that we have estimates on $U$ given in \ref{est}.
In the first case in \eqref{mono} we have both $A<0$ and $\frac{B(\beta -2)}{\beta}<0$. Since $U$ and $U^{\prime}$ are bounded we get
for $\epsilon$ small enough $G_{\epsilon}^{\prime}(t)<0$ for all $t$. In the fourth case, $\beta=2$ and $A=\la - \frac{N(N-2)}{4} >0$
so for $\epsilon$ small enough $G_{\epsilon}^{\prime}(t)>0$ for all $t$. In the third case $\beta<2, B<2$ and $A\ge 0$, again the 
conclusion follows. It remains to establish the claim in the second case in \eqref{mono}. In this case we have $A<0$ and $\beta<2, B<2$.
Hence $G_{\epsilon}^{\prime}(t 0$ for large enough $t$ and $G_{\epsilon}^{\prime}(t)>0$ for $t$ close to $0$. Thus it remains to show that 
$\left[ \frac{G_{\epsilon}^{\prime}(t)}{\beta (\sinh t)^{\beta -3} \cosh t} \right]^{\prime} <0$ for all $t$ if $\epsilon$ is small enough.
\begin{eqnarray*}
 \left[ \frac{G_{\epsilon}^{\prime}(t)}{\beta (\sinh t)^{\beta -3} \cosh t}    \right]^{\prime} && = 
 \left[A - {\epsilon} U^{p-1}\right]
2 \sinh t \cosh t - \frac{\epsilon}{\beta}(p-1)U^{p-2}U^{\prime}(t) (\sinh t)^{2} \\
&& - \frac{\epsilon(p-1)}{\beta}\left[ (p-2)U^{p-3}{U^{\prime}}^2 (\sinh t)^{2} \tanh t + U^{p-2}U^{\prime\prime}(t)(\sinh t)^{2} \tanh t  \right.\\
&& +  2 \left. U^{p-2} U^{\prime}(t)  (\sinh t \cosh t) \tanh t + U^{p-2} U^{\prime} (\sinh t)^{2} (\sech t)^{2} \right]
\end{eqnarray*}
From \eqref{est} and the equation $U$ satisfies, we get
$$\left[ \frac{G_{\epsilon}^{\prime}(t)}{\beta (\sinh t)^{\beta -3}\cosh t}  \right]^{\prime}
=\left[A + O(\epsilon)\right]2 \sinh t \cosh t < 0 \;{\rm for}\; \epsilon \;{\rm small \; enough}.$$ 
This completes the proof of claim.\\\\
Next we have the following comparison lemma for solutions of \eqref{E:3.4}, we again refer to \cite{MS}, Lemma 4.1 for its proof.\\
{\it Comparison Lemma.} Let $u,v$ be distinct positive solutions of \eqref{E:3.4} then\\
(i) given $R,M >0$ there exists $\delta= \delta(u,R)$ such that if $u(0),v(0) \le M$ then
$$ u(t_i) =v(t_i),\;0<t_1<t_2\le R, \Longrightarrow t_2-t_1 \ge \delta.$$
In addition if $u$ satisfies the finite energy condition and $v(0)<u(0)$ then there is $t_v$ such that 
$ u(t_v) =v(t_v)$ and there is $\tilde t  =\tilde{t}(u)$ such that
$$u(t_1) =v(t_1),\;0<\tilde{t}<t_1 ,\Longrightarrow v(t)>u(t) \;\forall t >\tilde{t} $$
\\
Having established, the estimates on the energy \eqref{energy0}-\eqref{energyinf}, the monotonicity properties of the energy \eqref{mono}
and the above comparison lemma, we get the following uniqueness result in geodesic balls:\\\\
Let $\la$ satisfy \eqref{conditions} then for every $T>0$ the Dirichlet problem
\begin{eqnarray*}
 &&u^{\prime\prime} + \frac{N-1}{\tanh t} u^{\prime} + \la_{\epsilon} u + u^{p} =0 \\
  &&u^{\prime}(0) = 0,\; u(T)=0,\;u(t)>0 \;\forall t\in [0,T).
 \end{eqnarray*}
has at most one solution.\\\\
The proof follows exactly like the case of $\epsilon =0$ and so we refer to \cite{MS}.
Using this uniqueness result of geodesic balls, the rest of the uniqueness proof follows exactly as in \cite{MS}.
\end{proof}
{\bf Proof of Theorem \ref{morseindex}}. From Theorem \ref{exis-per} and Theorem \ref{uni-per} we know that $I_{\epsilon}$ has a unique 
critical point and it is a 
mountain pass critical point. On the other hand $U$ is a critical point of $I_{\epsilon}$, and hence it is the unique critical point and of 
mountain pass type and hence the radial Morse index is one (See \cite{MP}).

%%%%%%%%%%%%%%%%%%%%%%%%%%%%%%%%%%%%%%%%%%%%%%%%%%%%%%%%%%%%%%%%%%%%%%%%%%%%%%%%%%%%%%%%%%%%%%%%%%%%%%%

\end{document}